\documentclass[a4paper,10pt]{article}

\usepackage[utf8x]{inputenc}
\usepackage[english]{babel}

\usepackage{amsmath}
\usepackage{graphicx}
\usepackage{hyperref}

\usepackage{amsfonts}
\usepackage{amsthm}
\usepackage{amssymb}
\usepackage{cite} 

\newcommand{\R}{\mathbb{R}} 
\newcommand{\N}{\mathbb{N}} 
\newcommand{\C}{\mathbb{C}}

\newcommand{\scap}{\operatorname{cap}}

 \newcommand{\Halb}{\mathbb{H}} 
\newtheorem{theorem}{Theorem}[section]
\newtheorem{lemma}[theorem]{Lemma}

\theoremstyle{definition}
\newtheorem{definition}[theorem]{Definition}
\newtheorem{remark}[theorem]{Remark}

\numberwithin{equation}{section}

\title{Improved long time existence for the Willmore flow of surfaces of revolution with Dirichlet data.}

\author{Sascha Eichmann\thanks{The author thanks Prof. Reiner Sch\"atzle for fruitful comments and Manuel Schlierf for discussing the equivalence of \cite[Thm. 1.3]{SchlierfWillmoreFlowImprov24} and Thm. \ref{1_1}.},\\
Mathematisch-Naturwissenschaftliche Fakultät,\\
Eberhard Karls Universität Tübingen,\\
Auf der Morgenstelle 10,\\
D-72076 Tübingen, Germany\\
E-mail: \href{mailto:sascha.eichmann@math.uni-tuebingen.de}{sascha.eichmann@math.uni-tuebingen.de}}
\begin{document}
 \maketitle

\begin{abstract}
To avoid possible singularities in the Willmore flow, one usually works under an energy threshold provided by the Li-Yau inequality. 
Here we improve this threshold by also considering parts outside of a possible singularity together with Dirichlet boundary data. We work in the class of surfaces of revolution.
\end{abstract}

\section{Introduction}
\label{sec:1}
For a smooth two dimensional oriented immersion $f:\Sigma\rightarrow\R^3$  with pull back measure $\mu_f$ and mean curvature $H$ (i.e. the sum of the principal curvatures) the Willmore energy is defined by
\begin{equation}
\label{eq:1_1}
 W_e(f) = \frac{1}{4}\int_\Sigma H^2\, d\mu_f.
\end{equation}
Thomsen gave a first analysis in \cite{Thomsen} of this energy. Further mathematical insight was not possible, because the necessary tools were not developed yet. 
Willmore later revived the discussion in \cite{Willmore} and formulated his famous conjecture for the optimal torus.
Work on this conjecture essentially began in \cite{Simon} and the result was finally proven in \cite{NevesMarques}.
This mainly concernes closed Willmore surfaces, i.e. critical points of $W_e$. 
Here we are examining the corresponding $L^2$-flow, which is given by
\begin{equation}
 \label{eq:1_2}
 \partial_t f = -(\Delta_{g_f} H + 2H(H^2-K)) N.
\end{equation}
$N$ is the unit normal, such that the vector valued mean curvature $\vec{H}$ satisfies $H=\vec{H}\cdot N$. 
Furthermore $K$ is the Gauss curvature and $\Delta_{g_f}$ is the Laplace-Beltrami operator of $f$.
In the series of papers \cite{KuwertSchaetzleWillmoreFlow02}\cite{KuwertSchatzle1}\cite{KuwertSchaetzle01}\cite{KuwertSchatzle2} Kuwert\&Schätzle worked out a lot of the analytical properties of this flow for closed immersions, which are also summarized in \cite{KuwertSchaetzleSurvey}.
Furthermore properties of the limit of the flow for approaching the maximal time of existence were found in \cite{ChillFasSchaetz09} for general closed immersions.
Contrasting this development we are interested in the $L^2$-flow with Dirichlet boundary data here (results for other  boundary data can be found in e.g. \cite{MenzelPHD2020} for the Navier case). 
To be precise let $f_0:\Sigma\rightarrow\R^3$ be given. 
Let $\nu_{f_0}:\partial\Sigma\rightarrow\R^3$ be the inner conormal of $f_0$. 
Then we examine the following initial/boundary value problem for a family of immersions $f:[0,T)\times \Sigma$ with $T\in[0,\infty]$.
\begin{equation}
\label{eq:1_3}
  \left\{\begin{array}{ll}\partial_t f = -(\Delta_{g_f} H + 2H(H^2-K)) N& \mbox{ in }[0,\infty)\times\Sigma\\
  f(0,p) = f_{0}(p) &\mbox{ for all }p\in\Sigma\\
  f(t,p) = f_{0}(p)& \mbox{ for }t\geq 0\mbox{ and }p\in\partial\Sigma\\
  \nu_{f(t,p)} = \nu_{f_{0}(p)} &\mbox{ for }t\geq 0\mbox{ and }p\in\partial\Sigma.
\end{array}\right.
 \end{equation}
 To extract more analytical properties, 
we additionally assume, that the initial data is given by a surface of revolution. 
They are defined by a smooth curve (also called profile curve)
$\gamma:[0,1]\rightarrow \R\times(0,\infty)=:\Halb$, which induces such a surface of revolution $S(\gamma)$ by the parametrisation $f_\gamma$:
\begin{equation}
 \label{eq:1_4}
[0,1]\times\mathbb{S}^1\ni(s,\varphi)\mapsto f_\gamma(s, \varphi) = (\gamma^1(s), \gamma^2(s)\cos\varphi, \gamma^2(s)\sin\varphi).
\end{equation}
Existence for minimisers under Dirichlet boundary data without rotational symmetry was discussed in \cite{Schaetzle}.
To obtain more analytic properties and finer results one can additionally impose the aforementioned rotational symmetry.
This was done in e.g.
\cite{DallDeckGru}, \cite{DallFroehGruSchie}, \cite{EichmannGrunau} and \cite{EichmannSchaetzleWillmoreExist23} (see also the references therein).
For the flow
of rotational symmetric tori long time existence results have been achieved in e.g. \cite{DallAcquaMuellerSchaetzleSpener}.
In this the central observation is, that if the hyperbolic length of the profil curve (see \eqref{eq:2_8}) stays bounded during the evolution, the flow exists for all times and converges to a Willmore surface.
This is closely related to the flow of elastic curves (cf. \eqref{eq:2_7}). We mention here some results in hyperbolic space
\cite{DallAcquaSpener2017},\cite{DallAcquaSpener2018}, \cite{MuellerSpenerElasticFlow}, \cite{SchlierfHypFlowBlowUp23}
and in euclidean space \cite{RuppSpener2020}, \cite{DallAcquaLinPozzi2017} (please see also the references therein).

Schlierf showed in \cite{SchlierfWillmoreFlow} that this idea of bounded hyperbolic length transfers to Dirichlet boundary data as well.
He also gave a sufficient condition for this
hyperbolic length to stay bounded. 
To be precise this is an energy bound for the initial datum on the Willmore energy. 
For the readers convenience we cite the precise result in Theorem \ref{A_2}.
Let us make the notion of boundary values for rotational surfaces more precise, i.e. let
 $x_\pm\in\R$, $\alpha_\pm>0$ and $\beta_\pm\in\R$. 
 Now we fixate the Dirichlet boundary data for a profile curve $\gamma:[0,1]\rightarrow\mathbb{H}$ and set (cf. \cite[Eq. (1.6)]{EichmannSchaetzleWillmoreExist23})
 \begin{equation}
 \label{eq:1_5}
 \begin{array}{cc}
  (x_-,\alpha_-):=\gamma(0),&\ (x_+,\alpha_+)=\gamma(1),\\
  (\cos\beta_-,\sin\beta_-)=\frac{d_s\gamma(0)}{|d_s\gamma(0)|},&\ -(\cos\beta_+,\sin\beta_+)=\frac{d_s\gamma(1)}{|d_s\gamma(1)|}
  \end{array}
 \end{equation}
 The energy bound by Schlierf in \cite{SchlierfWillmoreFlow} can also be interpreted as follows:\\
 There are uniquely defined circles or vertical lines $\gamma_{circ}(x_\pm,\alpha_\pm,\beta_\pm)$ (see \eqref{eq:2_9} and \eqref{eq:2_10}), such that the surface of revolution w.r.t. to the profile curve
 \begin{equation}
 \label{eq:1_7}
   \gamma_{closed} = \gamma_{circ}(x_-,\alpha_-,\beta_-)\oplus \gamma\oplus\gamma_{circ}(x_+,\alpha_+,\beta_+),
 \end{equation}
 i.e. $S(\gamma_{closed})$, is a close surface, see Figure \ref{fig:1_1}.
 \begin{figure}[h]
 \centering 
\includegraphics{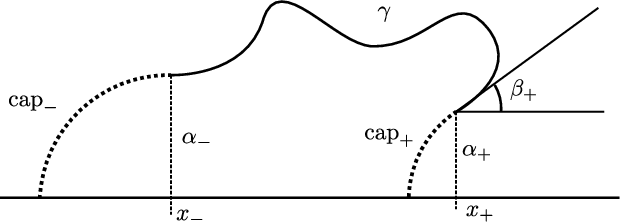}  
\caption{Boundary data and added sphere caps.}
\label{fig:1_1}
\end{figure}
 Furthermore we call the resulting sphere caps
 $$\scap_\pm := S(\gamma_{circ}(x_\pm,\alpha_\pm,\beta_\pm)).$$
Now we define the closed Willmore energy (cf. \cite[Eq. (2.12)]{EichmannSchaetzleWillmoreExist23})
\begin{align}
 \label{eq:1_8}
 \begin{split}
 W_{closed}^e(S(\gamma)) :=& W_e(S(\gamma)) + W_e(\scap_-) + W_e(\scap_+) + 4\pi\theta^2(\mathcal{H}^2\lfloor S(\gamma),\infty)\\
 & + 4\pi\theta^2(\mathcal{H}^2\lfloor \scap_-,\infty) + 4\pi\theta^2(\mathcal{H}^2\lfloor \scap_+,\infty).
 \end{split}
\end{align}
Here $\theta^2(\mathcal{H}^2\lfloor \scap_\pm,\infty)$ is the $2$-density of $\scap_\pm$ at infinity. The densities at $\infty$ are added, so that this energy becomes invariant for all inversions at  points in $\R\cup\{\infty\}$, see the argument in \cite[p. 6]{EichmannSchaetzleWillmoreExist23}, i.e. \cite{Chen74}. 
Now Schlierfs energy bound (see Theorem \ref{A_2}) is equivalent to the condition
\begin{equation}
 \label{eq:1_8_1}
 W^e_{closed}(S(\gamma_0)) \leq 8\pi
\end{equation}
for the initial curve $\gamma_0$ (combine \eqref{eq:2_7} with Lemma \ref{4_3}). This is the same bound given by the Li-Yau inequality (see \cite{LiYau}) for having selfintersections.
In other words Schlierfs bound mitigates a self intersection but does not incorporate the parts outside of this singularity.
Here we are going to exploit exactly these parts and improve the energy condition by Schlierf. 
These conditions cannot be improved arbitrarily. Counterexamples were provided by Blatt in \cite{Blatt2009} for closed surfaces and Schlierf in \cite{SchlierfWillmoreFlow} under Dirichlet boundary data. \\
Let us formulate our new energy condition. For this let $x\in\R\cup\{\infty\}$ be arbitrary but fixated. Let $c_\pm^x:[0,1]\rightarrow\overline{\Halb}$ be smooth and satisfy the boundary conditions at their respective starting points, i.e.
\begin{equation}
 \label{eq:1_9}
 \begin{array}{cc}
  (x_-,\alpha_-):=c_-^x(0),&\ (x_+,\alpha_+)=c_+^x(0),\\
  (\cos\beta_-,\sin\beta_-)=\frac{{d_sc_{-}^{x}}(0)}{|{d_sc_-^x}(0)|},&\ (\cos\beta_+,\sin\beta_+)=\frac{{d_sc_+^x}(0)}{|{d_sc_+^x}(0)|}
  \end{array}
 \end{equation}
such that 
\begin{equation}
\label{eq:1_10}
 c_\pm(1)=(x,0).
\end{equation}
Furthermore let $c_\pm^x$ be part of Moebius transformed catenoids or half circles with centre on the $x$-axis.
This is well defined for all $x\in\R\cup\{\infty\}$ by the Lemmata \ref{3_1} and \ref{3_2}.
Then we define the union $c^x:[0,2]\rightarrow\overline{\Halb}$ of these two curves as (see Figure \ref{fig:1_1_1})
\begin{figure}[h]
 \centering 
\includegraphics{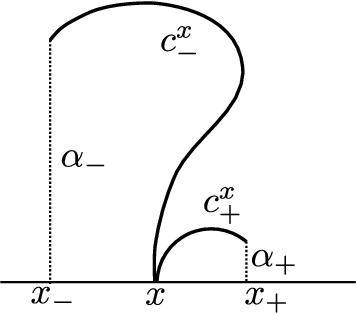}  
\caption{Sketch of $c^x$ consisting of one Moebius transformed catenoid and one sphere.}
\label{fig:1_1_1}
\end{figure}
\begin{equation}
 \label{eq:1_11}
 c^x(s):=\left\{\begin{array}{cc} c_-^x(s),& \mbox{ for }s\in[0,1]\\ c_+^x(2-s),& \mbox{ for }s\in[1,2]\end{array}\right.
\end{equation}
Similar pictures to $c^x$ also appear when the flow for the elastic energy (see \eqref{eq:2_6}) blows up, see \cite{SchlierfHypFlowBlowUp23}.\\
The main result of this paper is now as follows:
\begin{theorem}
 \label{1_1}
 Let $\gamma_0:[0,1]\rightarrow\Halb$ be a regular smooth curve satisfying the boundary data \eqref{eq:1_5}, such that
 $$W^e_{closed}(S({\gamma_0})) \leq \inf\{W^e_{closed}(S({c^x}))|\ x\in\R\cup\{\infty\}\}.$$
 Then there is a global, rotational symmetric solution $f:[0,\infty)\times\Sigma\rightarrow \R^3$ to \eqref{eq:1_4}, with initial data $\gamma_0$ and satisfying the boundary data \eqref{eq:1_5}.
 Moreover $f(t,\cdot)$ converges up to reparametrization smoothly to a Willmore immersion $f_\infty$ for $t\rightarrow\infty$. 
\end{theorem}
Independent and at the same time as our result Schlierf also published an improvement of the energy condition in \cite{SchlierfWillmoreFlowImprov24}. He used a Li-Yau inequality tailored to boundary data. 
Therefore this author conjectures that Schlierfs method can be used in a more general context, e.g. without rotational symmetry. 
Surprisingly the result in \cite[Thm 1.3]{SchlierfWillmoreFlowImprov24} and ours are exactly the same. 
Schlierf discusses this in detail in \cite[§ D]{SchlierfWillmoreFlowImprov24}.

The idea of the proof of Thm. \ref{1_1} is as follows: We will apply Schlierfs results from \cite{SchlierfWillmoreFlow}, which are summarized in Theorem \ref{A_3}. 
We proceed by contradiction and assume the hyperbolic length \eqref{eq:2_8} of the profile curves diverges to $\infty$.
Then by some geometric measure theory developed in \cite{ChoksiVeroni} we obtain a limit with a singularity on the $x$-axis and still obeying the boundary data.
Theorem \ref{4_7} essentially yields, that energy minimizers with boundary data and one singularity have to be Moebius transformed catenoids or circles. This is a contradiction to our energy assumption. 
The proof is given in detail in section \ref{sec:5}.
 
 Let us compare the energy condition by Schlierf \eqref{eq:1_8_1} to our condition given in Theorem \ref{1_1}. 
 Since any curve $c^x$ has exactly one singularity in the inner part, Lemma \ref{A_1} yields
 $$W^e_{closed}(S({c^x}))\geq 8\pi.$$
 Therefore our energy condition is always an improvement over \cite[Theorem 1.1]{SchlierfWillmoreFlow}, i.e. Theorem \ref{A_2}, if $c^x$ does not consists of two Moebius transformed circles. 
 In the circles case, the conditions are the same. 
 Furthermore the arguments in the proof of \cite[Prop. 3.9]{EichmannSchaetzleWillmoreExist23} together with our result guarantee the existence of admissible initial curves for almost all boundary data. 
 For the remaining cases \cite[Prop. 4.1]{EichmannSchaetzleWillmoreExist23} shows non-existence of energy minimising solutions, hence finding admissible initial curves is harder there (see also the discussion in \cite[§ 5]{SchlierfWillmoreFlow}).\\
 On the other hand, if the added sphere caps $\scap_\pm$ have a nontrivial intersection, the Li-Yau inequality (see \cite{LiYau}) yields
 $$W^e_{closed}(S({\gamma_0}))\geq 8\pi,$$
 which greatly impedes the applicability of \cite[Theorem 1.1]{SchlierfWillmoreFlow}.\\
 In the case of horizontal clamping, i.e. $\beta_+-\pi=\beta_-=0$, we calculate $W^e_{closed}(S({c^x}))$ in section \ref{sec:3}. 
 The following formula follows from combining \eqref{eq:3_22} with \eqref{eq:2_7}. 
 \begin{theorem}
  \label{1_2}Let $\alpha_+,\alpha_->0$ be given. If $\beta_-=0$, $\beta_+=\pi$ and $x_\pm=\pm1$, we have
  $$W^e_{closed}(S({c^x})) = 12\pi+ 4\pi\left(\frac{\alpha_+(x-1)}{\alpha_+^2 + (x-1)^2} - \frac{\alpha_-(x+1)}{\alpha_-^2 + (x+1)^2}\right).$$
  Furthermore we have the following asymptotic for fixated $\alpha_->0$
  $$\liminf_{\alpha_+\rightarrow\infty} \inf\{W^e_{closed}(S({c^x}))|\ x\in\R\cup\{\infty\}\}\geq 10\pi.$$
 \end{theorem}
 The asymptotic behaviour directly follows from \eqref{eq:3_24} and \eqref{eq:2_7}. 
 Since $x\mapsto W^e_{closed}(S({c^x}))$ is a rational function, the absolute minimiser can be calculated by numeric means by searching for a zero in the first derivative, which can be reduced to finding the root of a polynomial. 
 We collect some of these result in the following graphs, see Figures \ref{fig:1_2}, \ref{fig:1_3} and \ref{fig:1_4}.
 \begin{figure}[h]
 \centering 
\includegraphics{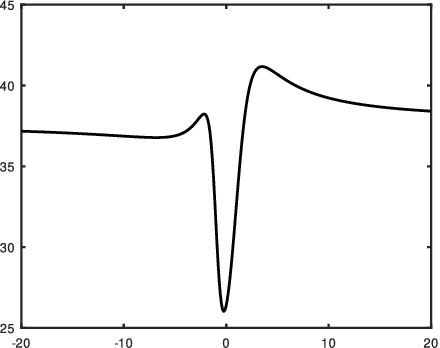}  
\caption{Plot of $x\mapsto W^e_{closed}(S({c^x}))$ für $\alpha_-=1$ and $\alpha_+=2$.}
\label{fig:1_2}
\end{figure}
\begin{figure}[h]
 \centering 
\includegraphics{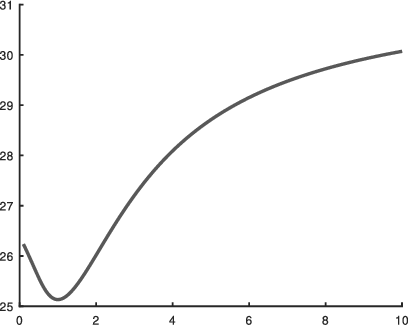}  
\caption{Plot of $\alpha_+\mapsto \inf_x W^e_{closed}(S({c^x}))$ for $\alpha_-=1$.}
\label{fig:1_3}
\end{figure}
\begin{figure}[h]
 \centering 
\includegraphics{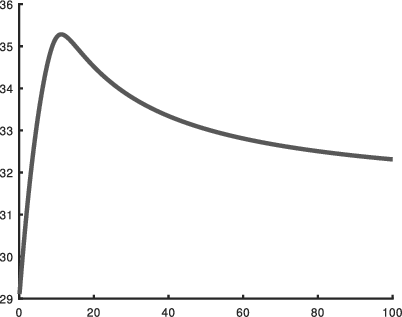}  
\caption{Plot of $\alpha_+\mapsto \inf_x W^e_{closed}(S({c^x}))$ for $\alpha_-=10$.}
\label{fig:1_4}
\end{figure}

\section{Geometric background}
\label{sec:2}
We collect some necessary geometric formulae for our calculations. These are from \cite{ChristianBaer} (cf. also \cite[§ 2.1]{EichmannPHD}):\\
Let $\gamma\in W^{2,2}([0,1],\R\times(0,\infty))$. We define the corresponding surface of revolution $S(\gamma)$ by \eqref{eq:1_4}. Then the metric tensor of this surface is given by
\begin{equation}
 \label{eq:2_1}
(g_{ij})_{i,j=t,\varphi} = \left(\begin{array}{cc} (d_s\gamma^1)^2 + (d_s\gamma^2)^2 & 0 \\ 0 & (\gamma^2)^2\end{array}\right).
\end{equation}
We introduce the Radon measure $\mu_\gamma$ on $\R$ (cf. \cite[Eq. 2.2]{ChoksiVeroni})
\begin{equation}
 \label{eq:2_2}
 \mu_\gamma:= 2\pi \gamma^2|d_s\gamma|\mathcal{L}^1\lfloor[0,1],
\end{equation}
which represents the integration on the surface $S(\gamma)$ in coordinates.
We call the principal curvatures of $S(\gamma)$ by $k_1^\gamma$, $k_2^\gamma$. Usually one would have to provide an orientation to calculate these object, but since all of our formula are independent of this choice, we can neglect this. For example the square of the scalar mean curvature is given as
\begin{equation}
 \label{eq:2_3}
 H^2 = (k_1^\gamma + k_2^\gamma)^2.
\end{equation}
Next we introduce the hyperbolic half plane $\Halb:=\{(x,y)\in\R^2|\ y>0\}$ equipped with the metric
\begin{equation}
 \label{eq:2_4}
ds^2=\frac{dx^2+dy^2}{y^2}.
\end{equation}
With respect to this metric (which we call $g_h$) we can define the geodesic curvature $\kappa_h[\gamma]$ (the $h$ stands for hyperbolic) of $\gamma$ by
\begin{equation}
 \label{eq:2_5}
 \kappa_h[\gamma] = \frac{g_h(\nabla_{d_s\gamma}d_s\gamma,N_\gamma)}{g_h(d_s\gamma,d_s\gamma)}.
\end{equation}
Here $N_\gamma$ is a choosen unit normal of $\gamma$ w.r.t. to $g_h$ and $\nabla_{\cdot}\cdot$ is the covariant derivative w.r.t. $g_h$. Then the (hyperbolic) elastic energy  (see e.g. \cite{LavienHistoryElastica} for a historical overview of elastic curves) is defined by
\begin{equation}
 \label{eq:2_6}
W_h(\gamma) := \int_0^1(\kappa_h[\gamma](t))^2\, ds(t).
\end{equation}
The Willmore energy and the elastic energy are connected by the following formula, which was first discovered in \cite{BryantGriffiths}.
\begin{equation}
 \label{eq:2_7}
\frac{2}{\pi} W_e(S(\gamma)) = W_h(\gamma) -  4\left[\frac{d_s\gamma^2}{\sqrt{(d_s\gamma^1)^2 + (d_s\gamma^2)^2}}\right]_0^1.
\end{equation}
If a curve $\gamma$ is critical w.r.t. to $W_h$, then $S(\gamma)$ also critical w.r.t. to $W_e$,
see e.g. \cite[Lemma 8.2]{PolyHarmBoundValue}.\\
Furthermore we call
\begin{equation}
 \label{eq:2_8}
 \mathcal{L}_{\Halb}(\gamma)=\int_0^1 |d_s\gamma(\ell)|\, ds(\ell)
\end{equation}
the hyperbolic length of the curve $\gamma$.

Let $x_0\in\R$, $\alpha_0>0$ and $\beta_0\in\R$. We define sphere caps $\scap(x_0,\alpha_0,\beta_0)$ (see Figure \ref{fig:1_1}) by rotating part of a circle (cf. \cite[Eq. 2.4]{EichmannSchaetzleWillmoreExist23})
 given by 
 \begin{equation}
  \label{eq:2_9}
  \begin{array}{c}\gamma_{circ}(x_0,\alpha_0,\beta_0):=\{(x,y)\in\Halb|\ |x-x_0 - \alpha_0 \tan\beta_0|^2 + y^2 = \frac{\alpha_0^2}{\cos^2\beta_0},\\ \cos\beta_0(x-x_0)\leq 0\}\end{array}
 \end{equation}
in case $\cos\beta_0\neq 0$. For $\cos\beta_0=0$ we put
\begin{equation}
 \label{eq:2_10}
 \gamma_{circ}(x_0,\alpha_0,\beta_0):=\{(x_0,y)|\ y>0,\ (\alpha_0-y)\sin\beta_0 > 0\},
\end{equation}
i.e. $\gamma_{closed}$ is a vertical line.

\section{Asymptotic geodesic elastica}
\label{sec:3}

Here we examine Moebius transformed catenoids $\gamma$. The classification in \cite[table 2.7 c)]{LangerSinger1} yields these to have hyperbolic curvature of the form
\begin{equation}
\label{eq:3_0_1}
 \kappa_h[\gamma](s)=\pm\frac{2}{\cosh(s+s_0)}
\end{equation}
for some $s_0\in\R$. Such curves are also called asymptotic geodesic (see e.g. \cite{LangerSinger1}). Furthermore they have finite elastic energy $W_h$, they approach the $x$-axis or $\infty$ and they are critical w.r.t. $W_h$. Hence they are prime candidates for rotational symmetric Willmore surfaces with a singularity.\\
Given an $s_0$ and an $\alpha>0$ we will solve the following Frenet equation in $\Halb$
\begin{equation}
 \label{eq:3_1}
 \left\{\begin{array}{c}\nabla_{d_s\gamma}d_s\gamma(s) = \pm \frac{2}{\cosh(s+s_0)} N_\gamma\\ \gamma(0)=(0,\alpha),\ d_s\gamma(0)=(\alpha,0)\end{array}\right.
\end{equation}
By Lemma \ref{3_1} this will exhaust all possibilities for Moebius transformed catenoids.
Further the solution $\gamma$ is parametrised by hyperbolic arclength.\\
The idea now is to use Moebius transformations of the upper half plane to find an explicit solution. Since we have two explicit solutions for $s_0=0$ and $\alpha=1$, we can find any solution. These special curves are given by
\begin{equation}
 \label{eq:3_2}
 s\mapsto (s,\cosh(s)),\ s\mapsto \frac{1}{s^2 + \cosh^2(s)}(s,\cosh(s)).
\end{equation}
This can be seen by direct computation as in e.g. \cite[§ 7.4]{EichmannDiplom}. \\
Any orientation preserving Moebius transformation of the upper half plane is given in complex coordinates $z\in\C$ by
\begin{equation}
 \label{eq:3_3}
 \varphi(z)= \frac{a z + b}{c z +d}
\end{equation}
with
\begin{equation}
 \label{eq:3_4}
 a,b,c,d\in\R,\ ad - bc > 0.
\end{equation}
It is also an isometry of the hyperbolic half plane (see e.g. \cite[Thm 4.6.2 and Thm. 4.6.7]{Ratcliffe}).
Now let $p\in\Halb$ and $V\in T_p\Halb$ with $|V|_g=1$. We now seek a Moebius Transformation $\varphi$ with
\begin{equation}
 \label{eq:3_5}
 \varphi(p)=(0,\alpha),\ \varphi'(p)(V) = (\alpha,0).
\end{equation}
If we have achived this, we will have solved \eqref{eq:3_1} by using \eqref{eq:3_2} to define $p$ as the point of the curve in $s_0$ and $V$ as the derivative of that curve in $s_0$. These type of computations have been done in e.g. \cite[Lemma 2.9]{MuellerSpenerElasticFlow}, but we need an explicit version here. Hence we repeat them for the readers convenience. (see e.g. \cite[Thm. 6.7]{EichmannDiplom} for a similar result by using inversions on spheres).\\
We write $p$ and $V$ in complex numbers as
$$ p = p_1 + i p_2,\ V=V_1 + i V_2.$$
The above requirements now mean
\begin{equation}
 \label{eq:3_6}
 p_2 >0\mbox{ and } V_1^2 + V_2^2 = p_2^2.
\end{equation}
To ease the calculations, we assume by composing with 
\begin{equation}
\label{eq:3_7}
  z \mapsto z - p_1,
\end{equation}
that $p_1=0$.
If $c\neq 0$, we can assume without loss of generality that $c=1$. Since $c=0$ corresponds to a linear transformation and since $a\in\R$ we would not loose any significant solutions apart from the ones already given in \eqref{eq:3_2}.
We now calculate
\begin{equation}
\label{eq:3_8}
 \varphi'(z)= \frac{a}{z+d} - \frac{az + b}{(z+d)^2} = \frac{az + ad -az -b}{(z+d)^2}=\frac{ad-b}{(z+d)^2}
\end{equation}
Now we have to satisfy the following condition
\begin{equation}
 \label{eq:3_9}
 \alpha = \varphi'(ip_2)(V_1 + i V_2).
\end{equation}
Seperating the right hand side in real and imaginary part yields
\begin{align*}
 &\varphi'(ip_2)(V_1 + i V_2)= (ad-b)\frac{V_1+ i V_2}{(ip_2 + d)^2}\\
 =& (ad-b)\frac{V_1 + iV_2}{d^2 + 2ip_2 d - p_2^2} = (ad-b)(V_1 + iV_2)\frac{d^2 - p_2^2 - 2ip_2 d}{(d^2 - p_2^2)^2 + 4 p_2^2 d^2}\\
 =& \frac{ad - b}{(d^2 - p_2^2)^2 + 4p_2^2d^2}(V_1 d^2 - V_1 p_2^2 + 2V_2 p_2 d +i(-2V_1p_2 d - V_2 p_2^2+ V_2 d^2))
\end{align*}
Since the prefactor is real and the imaginary part should vanish, we obtain by a solution formula for zeros of second order polynomials
\begin{align*}
 &0=-2V_1p_2 d - V_2 p_2^2+ V_2 d^2\\
 \Rightarrow& d^2 - 2\frac{V_1}{V_2}p_2 d - p_2^2 =0\\
 \Rightarrow& d_{1/2} = \frac{V_1}{V_2}p_2 \pm \sqrt{\frac{V_1^2}{V_2^2}p_2^2 + p_2^2} = \frac{V_1}{V_2}p_2 \pm \sqrt{\frac{V_1^2p_2^2 + V_2^2 p_2^2}{V_2^2}}\\
 &\overset{\eqref{eq:3_6}}{=} \frac{V_1 p_2}{V_2} \pm \sqrt{\frac{p_2^4}{V_2^2}}= \frac{V_1}{V_2}p_2\pm \frac{p_2^2}{|V_2|}.
\end{align*}
W.l.o.g. we can assume $V_2\neq 0$. Otherwise by scaling  \eqref{eq:3_2} would already yield our desired solution. \\
Now we have to decide a sign for this $\pm$: The real part of \eqref{eq:3_9} is supposed to be positive. Since the prefactor in the above calculation is supposed to be positive as well, we can employ the equation for $d$ and obtain
\begin{align*}
 0<& V_1 d^2 - V_1 p_2^2 + 2V_2 p_2 d = V_1\left(d^2 - p_2^2 + 2\frac{V_2}{V_1}p_2 d\right) = V_1\left(2\frac{V_1}{V_2} p_2 d + 2\frac{V_2}{V_1} p_2 d\right)\\
 =&  2p_2 V_1 d \left(\frac{V_1}{V_2} + \frac{V_2}{V_1}\right) = 2p_2^2V_1\left(\frac{V_1}{V_2}\pm \frac{p_2}{|V_2|}\right)\left(\frac{V_1}{V_2} + \frac{V_2}{V_1}\right)\\
 =& 2p_2^2 V_1\frac{V_1^2 + V_2^2}{V_2 V_1}\frac{V_1|V_2| \pm p_2 V_2}{V_2|V_2|} = 2p_2^2(V_1^2 + V_2^2)\frac{V_1|V_2| \pm p_2 V_2}{V_2^2|V_2|}
\end{align*}
By \eqref{eq:3_6} we have $|V_1 V_2|\leq p_2 |V_2|$. Hence we have to choose the sign as follows:
\begin{equation}
 \label{eq:3_10}
 d= \frac{V_1}{V_2}p_2 + \operatorname{sgn}(V_2) \frac{p_2^2}{|V_2|} = \frac{V_1 p_2 + p_2^2}{V_2}
\end{equation}
The condition on $p$ yields
\begin{align*}
 i\alpha =& \varphi(ip_2) = \frac{aip_2 + b}{ip_2 + d} = \frac{(aip_2 + b)(d-ip_2)}{d^2 + p_2^2}\\
 =&\frac{1}{d^2 + p_2^2}\left( ai p_2 d + ap_2^2 + bd - i b p_2\right).
\end{align*}
This yields the following two equations:
\begin{align*}
 \alpha (d^2 + p_2^2) =& ap_2 d - b p_2\\
 0=& a p_2^2 + bd.
\end{align*}
The second equation yields $a= - \frac{bd}{p_2^2}$ (please note $p_2>0$). By inserting this into the other equation, we obtain
$$\alpha(d^2 + p_2^2) = - \frac{bd^2}{p_2} - b p_2,$$
which in turn gives us
$$b = - \left(\frac{d^2}{p_2} + p_2\right)^{-1} \alpha (d^2 + p_2^2) = -\frac{p_2\alpha}{d^2 + p_2^2}(d^2 + p_2^2)=-p_2\alpha.$$
Since
\begin{align*}
& ad - b = -b\left(\frac{d^2}{p_2^2} + 1\right) >0,
\end{align*}
these $a,b,d$ satisfy \eqref{eq:3_4} and therefore we have found an eligible Moebius transformation.\\ 
Collecting all of our results and considering \eqref{eq:3_7}, we have
\begin{equation}
 \label{eq:3_11}
 d= \frac{V_1 p_2 + p_2^2}{V_2} - p_1,\ b = -\alpha p_2 - \frac{V_1  + p_2}{V_2}p_1\alpha,\ a = \frac{V_1  + p_2}{V_2}\alpha.
\end{equation}
Applying these results twice also yields the following lemma:
\begin{lemma}[cf. \cite{EichmannDiplom} Thm. 6.7, cf. \cite{MuellerSpenerElasticFlow} Lemma 2.9]
 \label{3_1}
 Let $q_1,q_2\in \Halb$ and $W_1\in T_{q_1}\Halb, W_2\in T_{q_2}\Halb$ with $|W_{1/2}|_g=1$. Then there exists an orientation preserving Moebius transformation $\Phi:\Halb\rightarrow\Halb$, i.e. an isometry of $\Halb$, such that
 $$\Phi(q_1)=q_2,\ d\Phi(q_1).W_1 = W_2.$$
\end{lemma}

Now we set
\begin{equation}
 \label{eq:3_12}
 p=(s_0,\cosh s_0),\ V=(1,\sinh(s_0))
\end{equation}
and calculate $\lim_{s\rightarrow\infty}\varphi\circ(s,\cosh(s))$, which will be the singularity on the $x$-axis we are looking for:
\begin{align*}
 \varphi(s,\cosh s)=\frac{a(s+i\cosh s) + b}{s+i\cosh s + d} \rightarrow a\mbox{ for }s\rightarrow\infty.
\end{align*}
This defines a function for the singularity $x(s_0)$:
\begin{equation}
 \label{eq:3_13}
 x(s_0) := \frac{1 +\cosh(s_0)}{\sinh(s_0)}\alpha.
\end{equation}
(Please note, that $s_0=0$, i.e. $V_2=0$, will be taken into account later).
Calculating the derivative yields
\begin{align*}
 d_{s_0}x(s_0)=&\alpha\left(\frac{\sinh(s_0)}{\sinh(s_0)} - (1+\cosh(s_0))\frac{\cosh(s_0)}{\sinh^2(s_0)}\right)\\
 =&\alpha\frac{\sinh^2(s_0) - \cosh(s_0) - \cosh^2(s_0)}{\sinh^2(s_0)} = -\alpha\frac{1+\cosh(s_0)}{\sinh^2(s_0)}<0.
\end{align*}
Hence it is strictly monotone and therefore invertible. Furthermore
\begin{equation}
 \label{eq:3_14}
 \lim_{s_0\rightarrow \pm 0}x(s_0)=\pm\infty,\ \lim_{s_0\rightarrow\pm\infty}x(s_0)=\pm\alpha.
\end{equation}
Hence we can reach any singularity in $(-\infty,-\alpha)\cup(\alpha,\infty)$.
Now applying an inversion at the sphere with center $0$ and radius $\alpha$ to $\varphi(s,\cosh(s))$ yields another elastica (see e.g. Figure \ref{fig_1}).
\begin{figure}[h]
 \centering 
\includegraphics{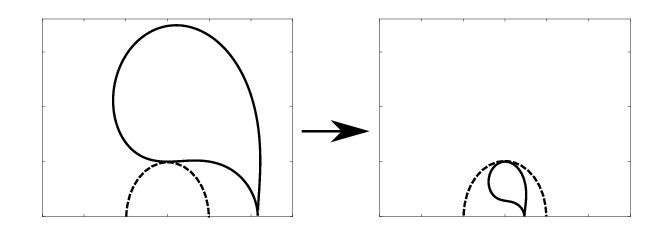}  
\caption{Inverting $\varphi(s,\cosh(s))$ on sphere with center zero and radius $\alpha$.}
\label{fig_1}
\end{figure}
The only change will be a change of sign in the geodesic curvature (cf. \eqref{eq:3_1}). Hence these elastica enable us to reach the interval $(-\alpha,\alpha)$. The points $\pm\alpha$ can be reached by half circles, which are geodesics and therefore elastica as well. Together with the uniqueness of solutions to the initial value problem for the Frenet equations and \cite[table 2.7 c)]{LangerSinger1} this yields
\begin{lemma}
 \label{3_2}
 For every $\alpha>0$ and every $x\in\R-\{\pm\alpha\}$ exists exactly one elastic curve $\gamma:\R\rightarrow\Halb$ with finite elastic energy, such that
 \begin{equation}
 \label{eq:3_15}
  \gamma(0)=(0,\alpha),\ d_s\gamma(0)=(\alpha,0),\ \lim_{s\rightarrow\infty} \gamma(s)=(x,0).
 \end{equation}
 This elastica is of asymptotic geodesic type, i.e. its curvature satisfies \eqref{eq:3_0_1}.\\
 Furthermore if $x=\pm\alpha$, there exists exactly one half circle connecting $(0,\alpha)$ with $(\pm\alpha,0)$, which is also an elastica.
\end{lemma}
\noindent Now we like to calculate the actual elastic energy of such asymptotic geodesics. Therefore we have to explictly compute $x^{-1}$ (Existence is guaranteed by $d_{s_0}x<0$, see above). Afterwards we will also obtain a formula for the inverted object, see Figure \ref{fig_1}. 
To this end \eqref{eq:3_13} yields
\begin{align*}
  0=&\alpha(1+\cosh(s_0)) - x(s_0)\sinh(s_0) = \alpha + \frac{\alpha}{2}\left(e^{s_0} + e^{-s_0}\right) - x(s_0)\frac{e^{s_0}- e^{-s_0}}{2}.
\end{align*}
Multiplying by $e^{s_0}$ yields
\begin{equation*}
 \alpha e^{s_0} + \frac{\alpha}{2}\left(e^{s_0}\right)^2 + \frac{\alpha}{2} - \frac{x(s_0)}{2}\left(e^{s_0}\right)^2 + \frac{x(s_0)}{2}=0.
\end{equation*}
By rearranging we obtain 
\begin{equation*}
 \frac{\alpha - x(s_0)}{2}\left(e^{s_0}\right)^2 + \alpha e^{s_0} + \frac{\alpha+x(s_0)}{2} =0. 
\end{equation*}
By \eqref{eq:3_14} we can assume $\alpha\neq x(s_0)$ (This would be the case for the half sphere anyway) and then
\begin{equation*}
 \left(e^{s_0}\right)^2 + \frac{2\alpha}{\alpha-x(s_0)} e^{s_0} + \frac{\alpha + x(s_0)}{\alpha - x(s_0)}=0.
\end{equation*}
Completing the square then yields
\begin{align*}
 \left(e^{s_0}\right)_{\pm} =& -\frac{\alpha}{\alpha - x(s_0)}\pm\sqrt{\frac{\alpha^2}{(\alpha - x(s_0))^2} - \frac{\alpha + x(s_0)}{\alpha - x(s_0)}}\\
 =& -\frac{\alpha}{\alpha - x(s_0)}\pm\sqrt{\frac{\alpha^2 - \alpha^2 + x(s_0)^2}{(\alpha - x(s_0))^2}}\\
 =& -\frac{\alpha}{\alpha -x(s_0)} \pm\frac{|x(s_0)|}{|\alpha - x(s_0)|}. 
\end{align*}
Since $|x(s_0)| > \alpha$ and we also have to perform $\log$ on both sides of the equation, the positive solution is the only viable one. Hence
\begin{equation}
 \label{eq:3_16}
 e^{s_0} = -\frac{\alpha}{\alpha -x(s_0)} +\frac{|x(s_0)|}{|\alpha - x(s_0)|}.
\end{equation}
If $x(s_0)>0$, then
\begin{equation*}
 e^{s_0} =-\frac{\alpha}{\alpha -x(s_0)} + \frac{x(s_0)}{x(s_0) - \alpha} = -\frac{\alpha + x(s_0)}{\alpha - x(s_0)}. 
\end{equation*}
On the other hand if $x(s_0)<0$
\begin{equation*}
 e^{s_0} =-\frac{\alpha}{\alpha -x(s_0)} + \frac{-x(s_0)}{-x(s_0) + \alpha} = -\frac{\alpha + x(s_0)}{\alpha - x(s_0)}.
\end{equation*}
Hence all in all we obtain
\begin{equation}
 \label{eq:3_17}
 s_0(x)=\log\left(-\frac{\alpha + x}{\alpha - x}\right).
\end{equation}
Let us now turn to the inversion at a sphere with radius $\alpha$ and centre $0$. It is given by (see e.g. \cite[Eq. 4.1.2]{Ratcliffe})
\begin{equation*}
  z\mapsto \frac{\alpha^2}{|z|^2} z.
\end{equation*}
For any point on the $x$-axis ($z=(0,x)$) with $|x|>\alpha$ this yields
\begin{equation*}
 x_{inv}:= \frac{\alpha^2}{x}.
\end{equation*}
Hence the formula for the inverted asymptotic geodesic is
\begin{equation}
 \label{eq:3_18}
 s_0^{inv} := \log\left(-\frac{\alpha + x_{inv}}{\alpha - x_{inv}}\right) = \log\left(-\frac{\alpha + \frac{\alpha^2}{x}}{\alpha - \frac{\alpha^2}{x}}\right)
 = \log\left(-\frac{x + \alpha}{x-\alpha}\right).
\end{equation}
In both cases, the square of the geodesic curvature function is of the form
$$\kappa(s)^2 = \frac{4}{\cosh(s)^2},$$
hence the energies are
$$E_+(\alpha):=\int_{-\infty}^{s_0} \frac{4}{\cosh(s)^2}\, ds \mbox{ and } E_-(\alpha):=\int_{s_0}^{\infty} \frac{4}{\cosh(s)^2}\, ds$$
depending on the direction we start at the initial value. Evaluating the first integral yields
\begin{equation}
\label{eq:3_19}
 E_+(\alpha)=4\left[\frac{\sinh s}{\cosh s}\right]_{-\infty}^{s_0} = 4\frac{\sinh s_0}{\cosh s_0} + 4
\end{equation}
and similarly the second integral is
\begin{equation}
 \label{eq:3_20}
 E_-(\alpha) = 4\left[\frac{\sinh s}{\cosh s}\right]_{s_0}^{\infty} =4- 4\frac{\sinh s_0}{\cosh s_0}.
\end{equation}
Now we turn to the following boundary value problem for a given $x\in\R$ and $\alpha_\pm>0$, cf. Theorem \ref{1_2}.
\begin{equation}
 \label{eq:3_21}
 \left\{\begin{array}{c}\mbox{Find two asymptotic geodesics }c^x_\pm:\R\rightarrow\Halb\mbox{ such that }\\ c^x_-(0)=(-1,\alpha_-),\ d_s c^x_-(0)=(\alpha_-,0), \lim_{s\rightarrow\infty}c^x_-(s)=(x,0)\\
         c^x_+(0)=(1,\alpha_+),\ d_s c^x_+(0)=(\alpha_+,0), \lim_{s\rightarrow-\infty}c_+^x(s)=(x,0).
        \end{array}\right. 
\end{equation}
In our formulas above we will have to replace $x$ by $x\mp 1$ depending on whether we start at the right or the left side. Instead of $s_0$ we call the starting point $s_\pm$. Then the elastic energy $W^h_{\alpha_-,\alpha_+}(x):= W_h(c^x_-\oplus c^x_+)$ is given with the help of \eqref{eq:3_19} and \eqref{eq:3_20} by
\begin{align*}
 W^h_{\alpha_-,\alpha_+}(x)=& 4\frac{\sinh s_+}{\cosh s_+} + 4 + 4- 4\frac{\sinh s_-}{\cosh s_-}\\
 \overset{\eqref{eq:3_17}/\eqref{eq:3_18}}{=}& 8 + 4\frac{\left|\frac{\alpha_+ + (x-1)}{\alpha_+ - (x-1)}\right| - \left|\frac{\alpha_+ - (x-1)}{\alpha_+ + (x-1)}\right|}{\left|\frac{\alpha_+ + (x-1)}{\alpha_+ - (x-1)}\right| + \left|\frac{\alpha_+ - (x-1)}{\alpha_+ + (x-1)}\right|} - 4 \frac{\left|\frac{\alpha_- + (x+1)}{\alpha_- -(x+1)}\right| - \left|\frac{\alpha_- - (x-1)}{\alpha_- + (x+1)}\right|}{\left|\frac{\alpha_- + (x+1)}{\alpha_- -(x+1)}\right| + \left|\frac{\alpha_- - (x-1)}{\alpha_- + (x+1)}\right|}\\
 =& 8 + 4\frac{(\alpha_+ + (x-1))^2 - (\alpha_+ - (x-1))^2}{(\alpha_++(x-1))^2 + (\alpha_+ - (x-1))^2}\\
 &- 4\frac{ (\alpha_- + (x+1))^2 - (\alpha_--(x-1))^2}{(\alpha_- + (x+1))^2 + (\alpha_- -(x+1))^2}\\
 =& 8 + 4\left(\frac{4\alpha_+(x-1)}{2\alpha_+^2 + 2(x-1)^2} - \frac{4\alpha_-(x+1)}{2\alpha_-^2 + 2(x+1)^2}\right).
\end{align*}
Hence the elastic energy is given by 
\begin{equation}
 \label{eq:3_22}
 W^h_{\alpha_-,\alpha_+}(x) = 8+ 8\left(\frac{\alpha_+(x-1)}{\alpha_+^2 + (x-1)^2} - \frac{\alpha_-(x+1)}{\alpha_-^2 + (x+1)^2}\right).
\end{equation}
By inspection this formula still holds true, if we use a half circle instead of an asymptotic geodesic, i.e. if $x=\pm 1 \mp \alpha_\pm$.
Hence by Lemma \ref{3_2}  the elastic energy of a solution to \eqref{eq:3_21} has to be given by \eqref{eq:3_22}.\\
We define the infimum to be
\begin{equation}
 \label{eq:3_23}
 W^{hmin}_{\alpha_-,\alpha_+} :=\inf\{W^h_{\alpha_-,\alpha_+}(x)|\ x\in\R\cup\{\infty\}\}. 
\end{equation}
Please note, that the case $x=\infty$ implies here $c^x_\pm$ to be of the form $s\mapsto \alpha_\pm\cosh(s/\alpha_\pm)$.\\
Now we also prove the claim for the asymptotic in Theorem \ref{1_2}, i.e. for $\alpha_->0$ fixated we have
\begin{equation}
 \label{eq:3_24}
 \liminf_{\alpha_+\rightarrow\infty} W^{hmin}_{\alpha_-,\alpha_+} \geq 4.
\end{equation}
\begin{proof}
Using a half circle and one asymptotic geodesic, we obtain by Lemma \ref{3_2}
$$W^{hmin}_{\alpha_-,\alpha_+} < 8.$$
 This and 
 $$\lim_{x\rightarrow\pm \infty} W^h_{\alpha_-,\alpha_+}(x)=8$$
yields that there exists an $x_{\alpha_+}\in\R$ such that
$$W^{hmin}_{\alpha_-,\alpha_+}=W_{\alpha_-,\alpha_+}(x_{\alpha_+}).$$
We distinguish two cases. The first is
\begin{equation}
 \label{eq:3_25}
 \limsup_{\alpha_+\rightarrow\infty}|x_{\alpha_+}|=\infty.
\end{equation}
After extracting a subsequence, we can actually assume to have convergence to $\infty$ in \eqref{eq:3_25}. Then Youngs inequality yields
\begin{align*}
 W_{\alpha_-,\alpha_+}(x_{\alpha_+}) =& 8+ 8\left(\frac{\alpha_+(x_{\alpha_+}-1)}{\alpha_+^2 + (x_{\alpha_+}-1)^2} - \frac{\alpha_-(x_{\alpha_+}+1)}{\alpha_-^2 + (x_{\alpha_+}+1)^2}\right)\\
 \geq & 8 - 4 - 8\frac{\alpha_-(x_{\alpha_+} +1)}{\alpha_-^2-(x_{\alpha_+}+1)^2}\rightarrow 4\mbox{ für }\alpha_+\rightarrow\infty,
\end{align*}
which handles the first case. The other case is
\begin{equation}
 \label{eq:3_26}
 \limsup_{\alpha_+\rightarrow\infty}|x_{\alpha_+}|<\infty.
\end{equation}
Hence we find a constant $C>0$, such that
$$x_{\alpha_+} \in[-C,C].$$
Hence we only need to examine $x\in[-C,C]$. Since we locally have uniform convergence, i.e.
$$W_{\alpha_-,\alpha_+}\rightarrow E_\infty$$
with
$$E_\infty(x):=8-8\frac{\alpha_-(x+1)}{\alpha_-^2 + (x+1)^2},$$
it is sufficient to find the minimisers of $E_\infty$:
Again by Youngs inequality, we have
$$E_\infty(x)\geq 8- 4\frac{\alpha_-^2 + (x+1)^2}{\alpha_-^2 + (x+1)^2}=4.$$
All in all this yields \eqref{eq:3_24}.
\end{proof}

\section{Minimising with a singular boundary value}
\label{sec:4}

\subsection{Variational Background}
\label{sec:4_1}
We repeat the necessary definitions and theorems of \cite{ChoksiVeroni} here, on which we build our minimising scheme in section \ref{sec:4_2}. 
The same procedure has also been used in \cite[§ 3]{EichmannSchaetzleWillmoreExist23}.\\
We start with a generalisation of profile curves, so called generalized generators.
They allow its corresponding surface of revolution to touch itself on the $x$-axis:
\begin{definition}[see Def. 3 in \cite{ChoksiVeroni}]
 \label{4_1}
 We say that $\gamma:[0,1]\rightarrow \R^2$ is a generalized generator, if $\gamma$ is Lipschitz continuous, $|d_s\gamma|\equiv \ell(\gamma)$ and $\gamma^2(t)>0$ for $\mathcal{L}^1$-almost every $t\in(0,1)$, $d^2_s\gamma\in L^1_{Loc}(\{\gamma^2>0\};\R^2)$ exists in a weak sense and there exists a $C>0$, such that
 \begin{equation}
  \label{eq:4_1}
  \int_0^1 (k_1^\gamma)^2 + (k_2^\gamma)^2\, d\mu_\gamma < C.
 \end{equation}
Here $\ell(\gamma)$ denotes the length of the curve $\gamma$ w.r.t. to the euclidean metric. $\kappa^\gamma_{1/2}$ are the principal curvatures of the corresponding surface of revolution and $\mu_\gamma$ is given by \eqref{eq:2_2}.
\end{definition}
Please note, that in \cite{ChoksiVeroni} the revolution of the curve was performed around the $y$-axis, hence the coordinates have been changed here to reflect this.\\
We have the following regularity:
\begin{lemma}[see Lemma 3 in \cite{ChoksiVeroni}]
 \label{4_2}
 Let $\gamma$ be a generalized generator as in Definition \ref{4_1}. Then for any subinterval $[a,b]\subset[0,1]\cap\{\gamma^2>0\}$ we have
 \begin{equation*}
  \gamma\in W^{2,2}((a,b),\R^2)\mbox{ and } d_s\gamma \mbox{ has a unique extension to }C^0([a,b],\R^2). 
 \end{equation*}
\end{lemma}
Next we have the behaviour of the tangent on the $x$-axis:
\begin{lemma}[see Lemma 4 in \cite{ChoksiVeroni}]
 \label{4_3}
 Let $\gamma$ be a generalized generator as in Definition \ref{4_1}.
 Let $a,b\in[0,1]$ be such that $\gamma^2(a)=\gamma^2(b)=0$, $\gamma^2(t)>0$ for all $t\in(a,b)$. Then the limits of $d_s{\gamma}$ as $t\rightarrow a^+$ rsp. $t\rightarrow b^-$ exist and furthermore
 \begin{equation*}
  \lim_{t\rightarrow a^+}d_s{\gamma}^1(t) = \lim_{t\rightarrow b^-}d_s{\gamma}^1(t)=0.
 \end{equation*}
 Also either
 \begin{equation*}
   \lim_{t\rightarrow a^+}d_s{\gamma}^2(t) =\ell(\gamma),\  \lim_{t\rightarrow b^-}d_s{\gamma}^2(t)=-\ell(\gamma)
 \end{equation*}
or
\begin{equation*}
   \lim_{t\rightarrow a^+}d_s{\gamma}^2(t) =-\ell(\gamma),\  \lim_{t\rightarrow b^-}d_s{\gamma}^2(t)=\ell(\gamma)
 \end{equation*}
 holds.
\end{lemma}

Next we define the convergence for a sequence of generators suitable for our variational endeavor.
\begin{definition}[see Def. 4 and Def. 6 in \cite{ChoksiVeroni}]
 \label{4_4}
 Let $\gamma_n$ be a sequence of generalized generators and $\gamma$ be a generalized generator as in Definition \ref{4_1} as well. We say $\gamma_n$ converges weakly as generators to $\gamma$, if and only if the following holds
 \begin{equation}
  \label{eq:4_2}
  \gamma_n\rightarrow\gamma\mbox{ uniformely in }C^0([0,1],\R^2), 
 \end{equation}
\begin{equation}
 \label{eq:4_3}d_s{\gamma}_n\rightarrow d_s{\gamma}\mbox{ strongly in }L^2((0,1),\R^2),
\end{equation}
\begin{equation}
\label{eq:4_4}
 \mu_{\gamma_n}\rightarrow \mu_\gamma \mbox{ weakly as measures}
\end{equation}
\begin{equation}
 \label{eq:4_5}
 \sup_{n\in\N}\int_{[0,1]}|d^2_s{\gamma}_n|^2\, d\mu_{\gamma_n} < \infty,
\end{equation}
\begin{equation}
 \label{eq:4_6}
 \mbox{for all }\varphi\in C^\infty_c(\R,\R^2)\mbox{ we have }\lim_{n\rightarrow\infty}\int d^2_s{\gamma}_n\cdot\varphi\, d\mu_{\gamma_n} = \int d^2_s{\gamma}\cdot\varphi\, d\mu_\gamma 
\end{equation}
\end{definition}
Please note, that \eqref{eq:4_6} together with \eqref{eq:4_4} and \eqref{eq:4_5} is also called convergence of the measure function pair $(d^2_s{\gamma}_n,\mu_{\gamma_n})$ to $(d^2_s{\gamma},\mu_\gamma)$, see e.g. \cite[§ 4]{Hutchinson}.
Now we turn our attention to lower-semicontinuity:
\begin{theorem}[see Prop. 1 in \cite{ChoksiVeroni}]
 \label{4_5} 
 Let $\gamma_n$ be sequence of generalized generators converging in the sense of Definition \ref{4_4} to a generalized generator $\gamma$. Then
 \begin{equation*}
  \liminf_{n\rightarrow\infty} W_e(S(\gamma_n)) \geq W_e(S(\gamma)).
 \end{equation*}
\end{theorem}

At last we cite the necessary compactness result:
\begin{theorem}[see Def. 1, Eq. 49, Eq. 50 and Prop. 2 in \cite{ChoksiVeroni}]
 \label{4_6}
 Let $\gamma_n:[0,1]\rightarrow\R\times[0,\infty)$ with $\gamma_n\in C^1((0,1),\R^2)\cap W^{2,2}_{loc}((0,1),\R^2)$ with $\gamma^2_n(0)=\gamma^2_n(1)=0$, $|d_s\gamma_n(t)|=\ell(\gamma_n)$ and $\gamma_n^2(t)>0$ for all $t\in(0,1)$. Furthermore let
 \begin{equation}
  \label{eq:4_7}
  \sup_{n\in\N}W_e(S(\gamma_n)) < \infty
 \end{equation}
 and 
 \begin{equation}
  \label{eq:4_8}
  \sup_{n\in\N}\mu_{\gamma_n}(\R) < \infty.
 \end{equation}
 Then either there exists a subsequence (after relabeling and possibly after a translation) $\gamma_n$ which converges to a generalized generator $\gamma$ in the sense of Definition \ref{4_5} or there exists a point $(z,0)\in\R^2$ such that $\gamma_n$ converges strongly in $W^{1,2}((0,1),\R^2)$ to that point.
\end{theorem}
Please note, that \eqref{eq:4_8} is a uniform bound on the area of the surfaces $S(\gamma_n)$.

\subsection{Existence with one sided singular boundary}
\label{sec:4_2}
Here we examine the following boundary value problem:
Given $p\in\Halb$, $\beta\in[0,2\pi)$ and $x\in\R\cup\{\infty\}$. Then we define
\begin{equation}
 \label{eq:4_9}
 M_{p,\beta,x}:=\left\{\begin{array}{c} 
 \gamma\in W^{2,2}_{loc}((0,1),\Halb)|\ \gamma\in W^{2,2}(0,1-\delta),\Halb)\mbox{ for all }\delta>0\\
 \gamma(0)=p,\ \frac{d_s\gamma}{|d_s\gamma|}(0) = (\cos\beta,\sin\beta),\\
 \lim_{s\rightarrow 1}\gamma(t)=(x,0),\mbox{ if }x\in\R,\\
 \lim_{s\rightarrow 1}|\gamma(t)|=\infty,\mbox{ if }x=\infty,\\
 d_s\gamma\neq0\end{array}\right\}
\end{equation}

Now we show existence of a minimiser  w.r.t. the hyperbolic elastic energy $W_h$. 
The strategy is the same as in \cite[Prop. 3.7]{EichmannSchaetzleWillmoreExist23}.
\begin{theorem}[cf.  \cite{EichmannSchaetzleWillmoreExist23} Prop. 3.7]
 \label{4_7}
 There exists a $\gamma_{min}\in M_{p,\beta,x}$, such that
 $$W_h(\gamma_{min}) = \inf\{W_h(\gamma)|\ \gamma\in M_{p,\beta,x}\} < \infty.$$
 Furthermore $\gamma_{min}$ is a critical point of $W_h$ and either part of a Moebius transformed catenoid or part of a hyperbolic geodesic. 
\end{theorem}
\begin{proof}
 Let $\gamma_\ell\in M_{p,\beta,x}$ be a minimising sequence, i.e.
 $$\lim_{\ell\rightarrow\infty}W_{h}(\gamma_\ell) = \inf\{W_h(\gamma)|\ \gamma\in M_{p,\beta,x}\} < \infty.$$
 The lemmata \ref{3_1} and \ref{3_2} yield this infimum to be finite. 
 As in \eqref{eq:1_7} we add the sphere cap (see \eqref{eq:2_9}/\eqref{eq:2_10})
 $$\scap:= S(\gamma_{circ}(p,\beta))$$
 to $\gamma_\ell$ and call the resulting curves
 $$\tilde{\gamma}_\ell := \gamma_{circ}(p,\beta) \oplus \gamma_\ell,$$
 see Figure \ref{fig:4_1}.
 \begin{figure}[h]
 \centering 
\includegraphics{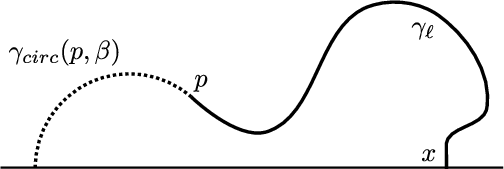}  
\caption{Adding a circle to close $\gamma_\ell$.}
\label{fig:4_1}
\end{figure}
Please note, that $\gamma_{circ}(p,\beta)$ does only depend on the boundary data $p$ and $\beta$ is therefore independent of $\ell$. 
 By an isometry of $\Halb$, i.e. an orientation preserving Moebius transformation, we can assume $\gamma_{circ}(p,\beta)$ to be of the form \eqref{eq:2_9}, i.e. part of a circle with center on the $x$-axis. 
 By the same procedure and the invariance of $W_h$ w.r.t. to Moebius transformations, we can assume w.l.o.g., that $x\in\R$.
 Again w.l.o.g. we reparametrize  $\tilde{\gamma}_\ell:[0,1]\rightarrow\overline{\Halb}$ and in such a way, that $\tilde{\gamma}_\ell$ is a generalized generator, see Definition \ref{4_1}. 
 This is possible because we will now show, that the corresponding Willmore energy $W(S(\tilde{\gamma}_\ell))$ is finite, see \eqref{eq:4_10}. Since the topology is fixated, e.g. \cite[Eq. (1.1)]{Schaetzle} then yields \eqref{eq:4_1} and $\tilde{\gamma}_\ell$ is a generalized generator.
 By $\tilde{\gamma}_\ell\in C^{1,1}$,   \eqref{eq:2_7} is applicable, which  yields for every $\varepsilon>0$
 $$W_e(S(\tilde{\gamma}_\ell|_{[\varepsilon, 1-\varepsilon]})) = \frac{\pi}{2}\left(W_h(\gamma_{circ}|_{[\varepsilon,1]}(p,\beta)) + W_{h}(\gamma_\ell|_{[0,1-\varepsilon]}) - 4\left[\frac{d_s\tilde{\gamma}^{2}_\ell}{|d_s\tilde{\gamma}_\ell|}\right]_{\varepsilon}^{1-\varepsilon}\right)$$
 after a suitable reparametrization of $\gamma_{circ}$ and $\gamma_\ell$.
 Since the $W_h$ terms are finite, the Willmore energy is finite as well and we can regard $\tilde{\gamma}_\ell$ as generalized generators. 
 Now Lemma \ref{4_3} is also applicable and we obtain
 \begin{align}\label{eq:4_9_1}
 \begin{split}
   W_e(S(\tilde{\gamma}_\ell)) =& \frac{\pi}{2}\left(W_h(\gamma_{circ}(p,\beta)) + W_{h}(\gamma_\ell) - 4\left[\frac{d_s\tilde{\gamma}^{2}_\ell}{|d_s\tilde{\gamma}_\ell|}\right]_{0}^1\right)\\
  =& \frac{\pi}{2} W_h(\gamma_\ell) + 4\pi.
 \end{split}
 \end{align}
 $W_h(\gamma_{circ}(p,\beta))=0$, because $\gamma_{circ}(p,\beta)$ is a geodesic w.r.t. the metric $g_h$, see \eqref{eq:2_4}. 
 Since the elastic energy of a full asymptotic geodesic is by \eqref{eq:3_19} rsp. \eqref{eq:3_20} exactly $8$, the Lemmata \ref{3_1} and \ref{3_2} yield after possibly choosing a subsequence
 \begin{equation}
  \label{eq:4_10}
  W_e(S(\tilde{\gamma}_\ell)) < \frac{\pi}{2} 8 + 4\pi = 8\pi.
 \end{equation}
In a first step we assume $\operatorname{diam}(S(\tilde{\gamma}_\ell))\leq C < \infty$, i.e. bounded independent of $\ell$. 
By \cite[Lemma 1]{Simon} and the bound on the Willmore energy \eqref{eq:4_10}, we obtain
$$\mu_{\tilde{\gamma}_\ell}(\R)=\operatorname{area}(S(\tilde{\gamma}_\ell)) \leq C<\infty.$$ 
Since the corresponding Willmore energy and the area are bounded, we can apply Theorem \ref{4_6} after a possible reparametrization.
Hence after choosing a suitable subsequence, we obtain a generalized generator $\tilde{\gamma}_{min}:[0,1]\rightarrow \overline{\Halb}$ (which is not a point on the $x$-axis due to the boundary data), such that
$$\tilde{\gamma}_\ell\rightarrow\tilde{\gamma}_{min}\mbox{ as generators.}$$
Next we show that 
\begin{equation}
 \label{eq:4_11_1}
 \forall\ t\in]0,1[ \mbox{ we have }\tilde{\gamma}_{min}^2(t) > 0.
\end{equation}
We proceed by contradiction and assume there is a $t\in(0,1)$, such that $\tilde{\gamma}_{min}^2(t)=0$. 
Then Lemma \ref{A_1} yields
\begin{equation}
 \label{eq:4_10_1}
 W_e(S(\tilde{\gamma}_{min}))=  W_e(S(\tilde{\gamma}_{min}|_{[0,t]})) + W_e(S(\tilde{\gamma}_{min}|_{[t,1]}))\geq 4\pi+4\pi=8\pi.
\end{equation}
On the other hand \eqref{eq:4_10} and the lower semicontinuity Theorem \ref{4_5} yield
$$W_e(S(\tilde{\gamma}_{min}))\leq \liminf_{\ell\rightarrow \infty}W_e(S(\tilde{\gamma}_\ell)) < 8\pi, $$
which is a contradiction. Hence there is no such $t$.

By the uniform convergence \eqref{eq:4_2}, we still have that $\tilde{\gamma}_{min}$ consists of the boundary circle $\gamma_{circle}(p,\beta)$ and some other part $\gamma_{min}:[0,1]\rightarrow\overline{\Halb}$ i.e. we have after a possible reparametrization
$$\tilde{\gamma}_{min}=\gamma_{circ}(p,\beta)\oplus \gamma_{min}.$$
By the Sobolov embedding and Lemma \ref{4_2} $\gamma_{min}$ satisfies the boundary data, hence
$$\gamma_{min}|_{(0,1)}\in M_{p,\beta,x}.$$
Finally the lower semicontinuity Theorem \ref{4_5} yields
\begin{align*}
 &W_e(S(\gamma_{circ})) + \liminf_{\ell\rightarrow\infty}W_e(S(\gamma_\ell))=\liminf_{\ell\rightarrow\infty}W_e(S(\tilde{\gamma}_\ell))\\
 \geq& W_e(S(\tilde{\gamma}_{min})
 = W_e(S(\gamma_{circ})) + W_e(S(\gamma_{min})).
\end{align*}
Hence by \eqref{eq:2_7}, i.e. \eqref{eq:4_9_1} we obtain
$$W_h(\gamma_{min}) \leq \liminf_{\ell\rightarrow\infty} W_h(\gamma_\ell),$$
hence $\gamma_{min}$ is a minimiser.\\
Let us now consider the case, that for a subsequence
$$\operatorname{diam}(S(\tilde{\gamma}_\ell))\rightarrow\infty.$$
By an inversion at a suitable circle with center $(x_{inv},0)$ on the $x$-axis (avoiding the intersection of $\gamma_{circ}$ with the $x$-axis and $x$ itself), we obtain a sequence of generalized generators obeying new boundary values.
By the conformal invariance of the Willmore energy and \eqref{eq:4_9_1} the new sequence is minimising for the new boundary data as well. Furthermore the new curves have bounded diameter.
Then the same arguments as above apply (especially \eqref{eq:4_10_1}, see also \cite[Figure 5]{EichmannGrunau} for a similar argument) and we obtain a minimiser, which does not intersect with the inversion point $(x_{inv},0)$ by \eqref{eq:4_11_1}.
Therefore going back to the original sequence yields by the uniform convergence \eqref{eq:4_2}, that the diameter is bounded. This is a contradiction.\\
All in all we found a minimiser $\gamma_{min}\in M_{p,\beta,x}$. It is obviously critical w.r.t. $W_h$ and by e.g. \cite[§5]{EichmannGrunau} it is smooth in $[0,1)$.
Moreover it has finite elastic energy $W_h$ and the discussion in \cite[table 2c]{LangerSinger1} yields $\gamma_{min}$ to be part of an asymptotic geodesic curve or a hyperbolic geodesic.
Asymptotic geodesic curves are Moebius transformed catenoids (cf. the discussion in section \ref{sec:3}) and therefore the result follows.
\end{proof}
\begin{remark}
 \label{4_8}
 Lemma \ref{3_2} together with Lemma \ref{3_1} yield the minimiser in Theorem \ref{4_7} to be unique.
\end{remark}

\section{Proof of main Theorem \ref{1_1}}
\label{sec:5}
We start with Theorem \ref{1_1}. Our argument relies on \ref{A_3}, i.e. the paper \cite{SchlierfWillmoreFlow}. Furthermore our argument is strongly inspired by \cite[Prop. 3.7 and Prop. 3.9]{EichmannSchaetzleWillmoreExist23}.
\begin{proof}
 Let $\gamma:[0,T)\times[0,1]$ be the family of profile curves corresponding to a short time solution of the rotational symmetric Willmore flow. Here $T\in[0,\infty]$ is the maximal time of existence.
 Existence and uniqueness is discussed in \cite[Appendix C]{SchlierfWillmoreFlow}, see also the references therein.
 We proceed by contradiction and assume the hyperbolic length diverges, i.e. that there exists a sequence $t_\ell\uparrow T$ with
 \begin{equation}
  \label{eq:5_1}
  \lim_{\ell\rightarrow\infty}\mathcal{L}_\Halb(\gamma(t_\ell,\cdot)) = \infty.
 \end{equation}
 As in the proof of Theorem \ref{4_7} or as in \eqref{eq:1_7} (see also Figure \ref{fig:1_1}) we complement the boundary conditions by adding sphere caps $\scap_- = S(\gamma_{circ}(x_-,\alpha_-,\beta_-))$ and $\scap_+ = S(\gamma_{circ}(x_+,\alpha_+,\beta_+)),$ to $\gamma(t_\ell,\cdot)$, which are independent of $\ell$. 
 We call the resulting curves $\tilde{\gamma}(t_\ell,\cdot)$.
 Since $W^e_{closed}$ is invariant under inversions, we assume the sphere caps to be bounded and spherical. We can that, because in the end we will just compare energies.\\
 Hence we are able to choose $x\in \R$, such that $c^x$ (see \eqref{eq:1_11}) consists of at least one part of a circle. Then \eqref{eq:3_19}/\eqref{eq:3_20} together with the Lemmata \ref{3_1} and \ref{3_2} yield
 $$W_h(c^x) < 8.$$
 Combining \eqref{eq:2_7} with Lemma \ref{4_3} gives us for this $x\in\R$ (cf. \eqref{eq:4_9_1})
 \begin{align}
  \label{eq:5_2}
  \begin{split}
    W^e_{closed}(S({c^x})) =& \frac{\pi}{2} (W_h(\gamma_{circ}(x_-,\alpha_-,\beta_-)) + W_h(\gamma_{circ}(x_+,\alpha_+,\beta_+))\\
    &+ W_h(c^x) + 8+8) < 12\pi.
  \end{split}
 \end{align}
 The assumptions on the initial data in Theorem \ref{1_1} therefore yields
 \begin{equation}
 \label{eq:5_2_1}
  W^e_{closed}(S({\gamma_0})) < 12\pi.
 \end{equation}
First we assume the diameter to be bounded, i.e.
$$\sup_{\ell}\operatorname{diam}S({\tilde{\gamma}(t_\ell,\cdot)}) < \infty.$$
Again using \cite[Lemma 1]{Simon} and the bound on the Willmore energy \eqref{eq:5_2_1}, we obtain
$$\mu_{\tilde{\gamma}(t_\ell,\cdot)}(\R)=\operatorname{area}(S(\tilde{\gamma}(t_\ell,\cdot))) \leq C<\infty.$$ 
As in the proof of Theorem \ref{4_7} the $\tilde{\gamma}(t_\ell,\cdot)$ are generalized generators.
Hence we can apply Theorem \ref{4_6} and choose a subsequence with
$$\tilde{\gamma}(t_\ell,\cdot)\rightarrow \tilde{\gamma}_{sing}\mbox{ as generators for }\ell\rightarrow\infty$$
and $\tilde{\gamma}:[0,1]\rightarrow\overline{\Halb}$ is a generalized generator. 
Since the sphere caps $\scap_\pm$ remain unchanged in the limit by the uniform convergence, \eqref{eq:4_2}, the inner regularity Lemma \ref{4_2} together with Sobolev embedding shows $\tilde{\gamma}_{sing}\in C^1_{loc}((0,1)\cap \tilde{\gamma}_{sing}^{-1}(\Halb))$.
Hence the restriction of $\tilde{\gamma}_{sing}$ to the part without the sphere caps, which we call after reparametrization $\gamma_{sing}:[0,1]\rightarrow\overline{\Halb}$, still satisfies the boundary values \eqref{eq:1_5}.
Then by \eqref{eq:5_2} and Lemma \ref{A_1} we find at most one $s_0\in (0,1)$ with
$$\gamma^2_{sing}(s_0) = 0.$$
Since the hyperbolic length of $\gamma(t_\ell,\cdot)$ diverges to $\infty$ and since the diameter is bounded, by uniform convergence, this $s_0$ has to exist.
Hence $\gamma_{sing}|_{[0,s_0]}$ and $\gamma_{sing}|_{[s_0,1]}$ respectively satisfy a boundary value problem as in Theorem \ref{4_7} with   $(\gamma_{sing}^1(s_0),0)$ as singular boundary value. Therefore Theorem \ref{4_7} yields 
$$ W^e_{closed}(S(\gamma_{sing})) \geq \inf\{W^e_{closed}(S({c^x}))|\ x\in \R\}.$$
The lower semi-continuity Theorem \ref{4_5}, the energy assumption of Theorem \ref{1_1} and the strict monotonicity of $t\mapsto W^e_{closed}(S({\gamma(t,\cdot)}))$ outside of critical points (see e.g. \cite[Eq. (2.3)]{SchlierfWillmoreFlow}) yields
\begin{align*}
 &W^e_{closed}(S(\gamma_{sing}))\leq \liminf_{\ell\rightarrow\infty}W^e_{closed}(S({\gamma(t_\ell,\cdot)})) < W^e_{closed}(S({\gamma(0,\cdot)}))\\
 \leq& \inf\{W^e_{closed}(S({c^x}))|\ x\in \R\} \leq W^e_{closed}(S(\gamma_{sing})), 
\end{align*}
which is a contradiction.\\
Let us now turn to the other case, i.e.
$$\operatorname{diam}(\tilde{\gamma}(t_\ell,\cdot))\rightarrow\infty\mbox{ for }\ell\rightarrow\infty.$$
By an inversion at an appropriate fixated point $(x_{inv},0)$ on the $x$-axis (cf. proof of Theorem \ref{4_7}), we obtain a sequence of curves with bounded diameter, different but fixated boundary values and the same Willmore energy. 
Hence we can apply the arguments of the first case and arrive at a contradiction there as well.
This is possible because $W^e_{closed}$ is invariant w.r.t. to Moebius transformations of the upper half plane and therefore the energy assumptions of Theorem \ref{1_1} carry over.\\
All in all we obtain a contradiction and the hyperbolic length stays bounded. Then Theorem \ref{A_3} yields the desired result.
\end{proof}

\appendix
\section{Helpful results}
\label{sec:A}
Here we gather some useful results from other publications for the readers convenience.\\
The first lemma is concerned with the needed Willmore energy for a singularity to occur.
A proof follows by combining Lemma \ref{4_3} and \eqref{eq:2_7}, but since this has already been done in \cite[Lemma 3.8]{EichmannSchaetzleWillmoreExist23} we skip the details here.
\begin{lemma}[see \cite{EichmannSchaetzleWillmoreExist23} Lemma 3.8]
 \label{A_1}
 Let $\gamma$ be a generalized generator as in Definition \ref{4_1}.
 Let $a,b\in[0,1]$ be such that $\gamma_2(a)=\gamma_2(b)=0$, $\gamma_2(t)>0$ for all $t\in(a,b)$. Then the Willmore energy of $S(\gamma|_{[a,b]})$ satisfies
 $$W_e(S(\gamma|_{[a,b]}))\geq 4\pi.$$
\end{lemma}

Next we cite the main result of \cite{SchlierfWillmoreFlow} to compare with our result.
\begin{theorem}[see \cite{SchlierfWillmoreFlow} Theorem 1.1]
 \label{A_2}
 Let $c_0:[0,1]\rightarrow\Halb$ be an immersed profile curve and $f_{c_0}$  the associated surface of revolution on $\Sigma:=[0,1]\times\mathbb{S}^2$, see \eqref{eq:1_3}. Furthermore we assume for the Willmore energy
 $$W_e(S({c_0})) \leq 4\pi - 2\pi\cdot \frac{d_sc^2_0}{|d_sc_0|}\bigg|_0^1.$$
 Then there is a global solution $f:[0,\infty)\times\Sigma$ of the Willmore flow under Dirichlet data , i.e. $f$ satisfies
 \begin{equation*}
  \left\{\begin{array}{ll}\partial_t f = -(\Delta_{g_f} H  +2H(H^2-K)) N& \mbox{ in }[0,\infty)\times\Sigma\\
  f(0,\cdot) = f_{c_0}(\cdot) &\mbox{ in }\Sigma\\
  f(t,p) = f_{c_0}(p)& \mbox{ for }t\geq 0\mbox{ and }p\in\partial\Sigma\\
  \nu_{f(t,p)} = \nu_{f_{c_0}(p)} &\mbox{ for }t\geq 0\mbox{ and }p\in\partial\Sigma.
\end{array}\right.
 \end{equation*}
 Moreover $f(t,\cdot)$ converges up to reparametrization smoothly to a Willmore surface $f_\infty$ for $t\rightarrow\infty$. Lastly $f_\infty$ and $f_t$ for all $t\geq0$ are embedded.
\end{theorem}

Next we cite a very useful result concerning long time existence for the Willmore flow under rotational symmetry. This result follows by combining \cite[Corollary 2.3]{SchlierfWillmoreFlow}, \cite[Theorem 4.5]{SchlierfWillmoreFlow} and \cite[Theorem 3.14]{SchlierfWillmoreFlow}.
\begin{theorem}
 \label{A_3}
 Let $c:[0,T)\times [0,1]\times\mathbb{S}\rightarrow\Halb$ be a time dependend profil curve, such that the corresponding surfaces of revolution $f_{c(t,\cdot)}$ satisfy the Willmore flow equation \eqref{eq:1_4}  under Dirichlet boundary conditions and let the solution be time maximal. If the hyperbolic length remains bounded
 $$\sup_{t\in[0,T)}\mathcal{L}_\Halb(c(t,\cdot)) < \infty$$
 the solution is global, i.e. $T=\infty$.\\
 Furthermore up to reparametrization the flow converges smoothly to a Willmore immersion.
\end{theorem}

\phantomsection
\addcontentsline{toc}{section}{References}
\bibliography{bibliography}
\bibliographystyle{plain}
\end{document}